\documentclass{amsart}
\usepackage{amssymb}
\usepackage{hyperref,url,biblatex}
\usepackage{graphicx} 
\usepackage{caption,comment}
\usepackage{cleveref}
\usepackage{tikz} 
\usepackage{quiver}
\usepackage[utf8]{inputenc} 
\usepackage[T1]{fontenc}

\def\wid{\operatorname{wd}}

\def\rk#1{\hbox{\rm rank}\,(#1)}

\usepackage{multicol,multirow}
\usepackage{color}
\definecolor{red}{rgb}{1.00,0.00,0.00}

\newtheorem{theorem}{Theorem}[section]
\newtheorem{lemma}[theorem]{Lemma}
\newtheorem{corollary}[theorem]{Corollary}
\newtheorem{proposition}[theorem]{Proposition}
\newtheorem{conjecture}[theorem]{Conjecture}
 
\newtheorem{remark}[theorem]{Remark}

\newtheorem*{acknowledgment*}{Acknowledgment}

\def\N{\mathbb{N}}

\def\Z{\mathbb{Z}}

\def\E{\mathbf{E}}

\def\b{{\bf b}}

\def\rk#1{\hbox{\rm rank}\,(#1)}

\def\N{\mathbb{N}}

\def\Z{\mathbb{Z}}

\newcommand{\D}{\Delta}

\renewcommand{\l}{\lambda}

\newcommand{\bs}{\setminus}
\newcommand{\q}{\quad}

\newcommand{\la}{\langle}
\newcommand{\ra}{\rangle}
\renewcommand{\ll}{\left\lfloor}
\newcommand{\rr}{\right\rfloor}

\newcommand*{\SM}[1]{S^e(#1)}
\newcommand*{\IM}[1]{I^e(#1)}
\newcommand*{\RM}[1]{R^e(#1)}
\newcommand*{\JM}[1]{\mathcal{J}_e(#1)}

\newcommand{\PF}{\text{PF}}
\renewcommand{\b}{\beta_{1,\l}}
\newcommand{\I}{\mathcal I}

\begin{document}
\title{Numerical Semigroups of Sally Type - II}

\author{Kriti Goel}
\address{University of Missouri, USA}
\email{kritigoel.maths@gmail.com}
\author{N\.{i}l \c{S}ah\.{i}n}
\address{Department of Industrial Engineering, Bilkent University, Ankara, 06800 Turkey}
\email{nilsahin@bilkent.edu.tr}
\author{Srishti Singh}
\address{University of Missouri, USA}
\email{spkdq@umsystem.edu}
\author{Hema Srinivasan}
\address{University of Missouri, USA}
\email{srinivasanh@missouri.edu}
\subjclass[2000]{}


  
\begin{abstract}
In this paper we study numerical semigroups of Sally type of multiplicity $e$ and embedding dimension $\nu \ge e-2$.  We construct the minimal resolutions for these semigroup rings when they are symmetric and compute their Betti numbers.  We also construct a minimal resolution for another special class of such semigroups of type $\nu-1$. Finally, we propose some conjectures for the Betti numbers of families of non-symmetric Sally type semigroups in the above cases in relation to those of the corresponding Gorenstein cases of Sally type semigroups. 
\end{abstract} 

\maketitle
\section{Introduction}
A numerical semigroup $S$ is called \textit{Sally type} if its multiplicity is one more than its width.  The study of Sally type numerical semigroups is inspired by a remarkable result of J. Sally \cite{Sally} where she introduces one such numerical semigroup as an example of a Gorenstein ideal of multiplicity two more than embedding dimension whose associated graded ring is Cohen-Macaulay. Thus, she establishes, for a general Gorenstein local ring of dimension $d$ and multiplicity $e$, the associated graded ring is guaranteed to be Cohen-Macaulay only if the embedding dimension is $d,$ or $d+1,$ or $e+d-2,$ or $e+d-3$.

Let $S=\langle s_1,s_2,\ldots,s_g \rangle=\{\sum_{i=1}^{g}u_is_i \mid u_i\in \mathbb{N}\}$  be a numerical semigroup, where $s_1<s_2<\cdots<s_g$ is a sequence of positive integers with $\gcd(s_1,s_2,\hdots,s_g)=1$. Let $k$ be a field, and $A = k[X_1, X_2, \dots, X_g] $ the polynomial ring in $g$ variables.   Consider the $k$-algebra homomorphism $ \phi: A \to k[S]$  defined by
$\phi(X_i) = t^{s_i}$, where  $k[S]=k[t^{s_1},\hdots,t^{s_g}]$ is the semigroup ring of $S$. The kernel of this homomorphism, denoted by $I_S,$ is the defining ideal of the affine monomial curve $C_S$ with parameterization $$X_1 = t^{s_1} , X_2=t^{s_2} , \ldots , X_g =t^{s_g}.$$

When $s_1,s_2,\ldots, s_g$ is a minimal set of generators for $S$, the difference $\wid(S) = s_g-s_1$ is the width of $S$, while $s_1$ is its multiplicity and $g$ is its embedding dimension. Let $R_S = k[[t^{s_1} , \ldots, t^{s_g} ]]$ be the local ring with the maximal ideal $\mathfrak{m} = \langle t^{s_1}, \ldots,t^{s_g} \rangle$. Then $gr_{\mathfrak{m}}(R_S )=\bigoplus_{i \ge 0} 
\mathfrak{m}_i/\mathfrak{m}_{i+1}\cong A/I^*_S$ is the associated graded ring where $I^*_S=\langle f^* \mid f\in I_S \rangle$ with $f^*$ denoting the initial form of $f$.

Let $[t,s]$ denote the set of non-negative integers $i$ with $t\le i\le s$. $\SM{m}$ denotes the Sally type numerical semigroup minimally generated by $e+i$ with $i\in [0,e-1]\setminus \{m\}$, and $\SM{m,n}$ denotes the Sally type numerical semigroup minimally generated by $e+i, i \in [0,e-1] \bs \{m,n\}$.  In \cite{DGS3},  Sally type numerical semigroups of multiplicity $e$ and embedding dimension $e-2$, $\SM{m,n}$ when $2\le m< n\le e-2$, are studied. Their defining ideals are described explicitly, and their types, Frobenius numbers, and first Betti numbers are computed.  In this paper, we add the cases when $n=e-1$, although that is not quite Sally-type with multiplicity $e$ as the width changes.  We also fill in the $m=1$ which is idiosyncratic in general, as can be seen in the result on the Cohen-Macaulay type of $\SM{m,n}$ (see Theorem \ref{thm:type m,n}). At the time of submission of the manuscript, it was brought to the attention of the authors that in \cite[Theorem 3.1]{BCJ}, a minimal generating set for the defining ideal of $\SM{1,n}$, for $n \in [2,e-2],$ is computed.

In this paper, we compute the type and Frobenius number for all $\SM{m}$ and $\SM{m,n}$, and construct minimal free resolutions for the corresponding semigroup rings of $\SM{1}$, $\SM{2}$, $\SM{2,3}$ and $\SM{3,4}$. We show that the type of $\SM{m}$ is $e-m$, except when $m=1$ or $e-1$, in which case it is symmetric and of type $1$. For $\SM{1,n}$, we show that it is never symmetric and has type $2$ or $3$, whereas $\SM{m,n}$ is symmetric precisely when $(m,n)=(2,3)$, the case studied by Sally. Furthermore, if $(m,n)\neq (2,3)$ and $m \geq 2$, the type of $\SM{m,n},$ is $e-n+1$ or $e-n+2$ depending on whether $n\le 2m$ or $n>2m$.  

Using Hochster’s formula, we compute the minimal number of generators of $\SM{m}$ (Theorem \ref{thm:first betti m}), following the approach of \cite{DGS3}, and also compute their defining ideal (Theorem \ref{thm:defining ideal m}). In Theorems \ref{thm:res of I1}, \ref{thm:res of Im}, \ref{thm:res of I23}, and \ref{thm:res of I34}, we construct minimal resolutions and determine the Betti numbers for the symmetric cases, $\SM{1}, \SM{e-1}, \SM{2,3}$ and for the special cases $\SM{2}$ and $\SM{3,4}$. As a consequence, Corollaries \ref{Betti1},  \ref{Betti2}, \ref{Betti3}, and  \ref{cor:I34} show that the Betti numbers of $\SM{1}, \SM{2}$, $\SM{2,3},$ and $\SM{3,4}$ all satisfy $\beta_t\le t{e\choose t+1},$ the conjectured bound of Caviglia-Moscariello-Sammartano \cite{CMS}:
$$\beta_t\leq t{\wid(S)+1 \choose t+1}.$$

Finally, we study the Betti numbers of Sally type semigroups, $\SM{m}$ and $\SM{m,n},$ as a function of $m$ and $n.$ For a fixed $m \geq 4,$ the family of numerical semigroups of the form $\SM{m,m+1}$  have remarkable similarities. For instance, it appears that the first $(m-3)$ Betti numbers are identical for $\SM{m,m+1}$ and $\SM{m+1,m+2}$. In Section \ref{sec:conj}, we propose some conjectures for the Betti numbers of families of Sally type semigroups in relation to those of the corresponding Gorenstein cases of Sally type semigroups.

\section{Structure of $\SM{m}$} \label{sec:drop1}

The goal of this section is to understand the structure of $\SM{m}$ in detail. In the first part, we study the numerical semigroup itself and compute its Frobenius number and Cohen–Macaulay type. In the second part, we examine the defining ideal, determining its height, the number of minimal generators, and an explicit minimal generating set. In the final part, we construct minimal free resolutions for the cases $m=1$ and $m=2.$

\subsection{The Numerical Semigroup $\SM{m}$.}

 We start by computing the Frobenius number $F(\SM{m})$ of $\SM{m}$, i.e., the largest integer not in $\SM{m}$. We then use this information to obtain the Cohen-Macaulay type of $\SM{m}$, which can be computed as the cardinality of the set of pseudo-Frobenius elements of $\SM{m}$ (\cite[Proposition 2.7]{cavaliere1982form}). 

\begin{proposition}
    For $1 \leq m \leq e-1$, the Frobenius number of $\SM{m}$ is given by 
    \[ F(\SM{m}) = 
    \begin{cases}
        2e+1, & m = 1, \\
        e + m, & \text{otherwise.}
    \end{cases}\]    
\end{proposition}    
\begin{proof}
    It is enough to check that this $F(\SM{m}) \notin \SM{m}$ and for all $s$ such that $F(\SM{m}) + 1 \leq s \leq F(\SM{m})+e,$ we have $s \in \SM{m}.$ While it is easy to see that $F(\SM{m}) \notin \SM{m},$ we check the latter condition in each case separately.

    Let $m = 1.$ Then $\SM{1} = \la e,e+2,\ldots,2e-1 \ra.$ For any $s$ such that $2e+2 \leq s \leq 3e-1,$ we have $s= e + (e+\alpha) \in \SM{1}$ where $2 \leq \alpha \leq e-1.$ Moreover, $s = 3e \in \SM{1}$ and $s = 3e+1 = (2e-1) + (e+2) \in \SM{1}.$

    For $m \neq 1,$ we show that $[e+m+1,2e+m] \subseteq \SM{m}.$ Clearly, $e+m+1,\ldots,2e-1\in \SM{m}.$ It remains to show that $2e+\alpha \in \SM{m}$ where $0 \leq \alpha \leq m.$ Note that for $\alpha \in [0,m-1],$ we have $2e+\alpha = e +(e+\alpha) \in \SM{m}.$ Also, $2e+m = (e+1) + (e+m-1) \in \SM{m}.$ 
\end{proof}

The \textit{type} of a numerical semigroup $S,$ denoted by $t(S),$ is the cardinality of the set $\{x \in \mathbb{Z} \setminus S \mid x + s \in S,  \forall s \in S \setminus \{0\}\}.$ Such a set is denoted by $\PF(S)$ and its elements are called pseudo-Frobenius numbers of $S.$
A numerical semigroup $S$ is \textit{symmetric} if there is an integer $F$ such that an integer $x\in S \Leftrightarrow F-x \notin S$.  This number $F$ is called the Frobenius number as it is the largest integer not in $S$.  Thus, to be symmetric, a numerical semigroup must have an odd Frobenius number. Note that being symmetric is equivalent to the associated semigroup ring being Gorenstein. Thus, the type of a symmetric semigroup is $1$ and the number of gaps is half of $F+1$. The numerical semigroup is \textit{almost-symmetric} if the number of gaps is half of $F+t(S).$ Thus, the following computation of the Cohen-Macaulay type of $\SM{m}$ also gives us a characterization of when it is symmetric or almost symmetric.

\begin{theorem} 
For $1 \leq m \leq e-1$, the Cohen-Macaulay type of $\SM{m}$ is
\begin{align*}
    t(\SM{m}) = \begin{cases}
        1, & m \in \{1,e-1\}, \\
        e-m, & \text{ otherwise.}
    \end{cases}
\end{align*}
In particular, $\SM{m}$ is symmetric if and only if $m \in \{1,e-1\}$ and is almost symmetric otherwise. 
\end{theorem}

\begin{proof}
If  $m = 1,$ then from the above result we know that the Frobenius number, $F(\SM{1}) = 2e+1.$ Denote by $G(\SM{m})=\{x \in \N \mid x \not\in \SM{m}\}$ the set of gaps of $\SM{m}$. Since $G(\SM{1})= [1,e-1] \cup \{e+1, 2e+1\}$ satisfies $|G(\SM{1})| = e+1 = (F(\SM{1})+1)/2,$ we get $\SM{1}$ is a symmetric numerical semigroup. It further follows that $\PF(\SM{1}) = \{F(\SM{1})\}$ and $t(\SM{1})=1.$ Similarly, $G(\SM{e-1})= [1,e-1] \cup \{2e-1\}$ and $|G(\SM{e-1})| = e = (F(\SM{e-1})+1)/2$ implies that $\SM{e-1}$ is a symmetric numerical semigroup with $t(\SM{e-1})=1.$ 

If $m \notin \{1,e-1\},$ then $G(\SM{m})=[1,e-1] \cup \{e+m\}.$ We claim that $\PF(\SM{m}) = [m+1,e-1] \cup \{e+m\}.$ Since, $e+m = F(\SM{m}),$ it is a pseudo-Frobenius number. For $x \in [m+1,e-1],$ as $x+(e+j) \geq e+m+1$ for all $j \in [0,e-1] \setminus \{m\},$ we get $x+(e+j) \in \SM{m}.$ Whereas, for $x \in [1,m],$ $x+(e+(m-x)) = e+m \notin \SM{m}.$ Since $m-x \in [0,e-1] \setminus \{m\},$ the claim follows. Therefore, when $m \notin \{1,e-1\},$ we have $t(\SM{m}) = e-m$ and $|G(\SM{m})| = e = (F(\SM{m})+t(\SM{m}))/2$ implying that $\SM{m}$ is almost symmetric (\cite[Proposition 10]{NumSgps}).
\end{proof}

\subsection{The Defining Ideal $\IM{m}$.}

For $1 \leq m \leq e-1$, let $\SM{m}$ be the numerical semigroup generated by the set  $[e, 2e-1] \bs \{e+m\}$. Let $\RM{m}=k[X_i \mid i\in [0,e-1] \bs \{m\}]$ denote a polynomial ring over a field $k$ with weighted degrees, $\deg X_i = e+i,$ for all $i \in [0,e-1] \bs \{m\}.$ Then the semigroup ring $k[\SM{m}] \cong \RM{m} / \IM{m}$ has embedding dimension $e-1,$ and multiplicity $e$, where $\IM{m}$ is the defining ideal of $\SM{m}.$ That is, $\IM{m}$ is a binomial prime ideal of height $e-2.$ Our study of $\IM{m}$ begins with determining its minimal number of generators; in other words, the first Betti number of $k[\SM{m}].$
\begin{theorem}  \label{thm:first betti m}
    The minimal number of generators of the defining ideal $\IM{m},$
    $$\mu(\IM{m}) =\begin{cases}
        {e-1 \choose 2}, & m=2, \\
        {e-1 \choose 2}-1, & m\neq 2.
    \end{cases}$$
\end{theorem}

\begin{proof}
    We prove this by computing the first Betti number of $k[\SM{m}]$ using Hochster's formula \cite[Theorem 5.5.1]{BrunsHerzog} and proceed in the following manner: For each $\lambda \in \SM{m}$, consider the simplicial complex $\Delta_\lambda$ and let $\delta_i$ denote the homogeneous differential maps of degree 0 in the minimal $\Z$-graded resolution of $\Delta_\l$, and compute
\begin{enumerate}
     \item[(I)] the number of zero-dimensional faces in $\D_\l$,
     \item[(II)] $\rk M$, where $M$ denotes the matrix representation of $\delta_1$,
     \item[(III)] $\beta_{1,\l} = \text{(I)}-\text{(II)}-1.$
\end{enumerate}
Then $\beta_1(\SM{m}) = \sum_{\l \in \SM{m}} \beta_{1,\l}.$ This works because all but finitely many $\b$ are zero. We begin by eliminating all $\l \in \SM{m}$ with $\beta_{1, \lambda}=0.$ In particular, we prove that $\beta_{1, \lambda}=0$ for all $\lambda \leq 2e+1$ and $\lambda \geq 4e-1$. Note that $$\D_\lambda = \begin{cases}
    \{\emptyset, \{\l\}\}, & \l \leq 2e-1, \\
    \{\emptyset, \{e\}\}, & \l =2e, \\
    \{\emptyset, \{e\},\{e+1\}, \{e,e+1\}\}, & \l=2e+1,
\end{cases}$$
which are all exact. Now, suppose that $\l \geq 4e-1$, and write $\l =4e-1+x$ for $x \geq 0$. When $m \neq 1$, $\{e,e+j\} \in \D_\l$ for all $j \in [1,e-1]\bs\{m\}$ unless
\begin{equation}\label{eq:biggerthan4e-1}
    e+x-j-1=m, \text{ i.e., } j=e+x-m-1 \in [1,e-1]\bs\{m\}.
\end{equation}
However, in this case we may pick $y \in [1,m-1]$ to get $\{e+y,e+(e+x-m-1)\}\in \D_\l$. \cite[Lemma 3.5]{DGS3} now implies $\beta_{1,\lambda}=0$ for $\l \geq 4e-1$ and $m \neq 1$. When $m=1$, $\{e,e+j\} \in \D_\l$ for all $j \in [1,e-1]\bs\{m\}$ unless \Cref{eq:biggerthan4e-1} is true or
\begin{equation}\label{eq:biggerthen4e-1form=1}
    e+x-j-1=e+1, \text{ i.e., } j=x-2 \in [1,e-1]\bs\{1\},
\end{equation}
since $F(\SM{1})=2e+1$. Putting $m=1$ in \Cref{eq:biggerthan4e-1} we note that $\{e+(e+x-2)\}$ is not a zero-dimensional face in this case and if \Cref{eq:biggerthen4e-1form=1} occurs, then pick $y \in \{e-1,e-2\}\bs\{j\}$ to obtain $\{e+y,e+j\} \in \D_\l$. Again, \cite[Lemma 3.5]{DGS3} implies $\beta_{1,\lambda}=0,$ establishing the claim.

Next, assume $2e+2 \leq \lambda \leq 4e-2$, $\lambda \neq 2e+2m$. Write $\l = 2e+2+x$, where $0 \leq x \leq 2e-4$, and $x \not\in \{2m-2\}.$ Then 
\begin{align*}
    \{e+j\} \in \D_\l \Longleftrightarrow 
    \begin{cases}
    j \in \I = \big([0,e-1] \cap [0,x+2]\big) \bs \{ m,x+2-m \}, & m \neq 1, \\
    j \in \I' = \big([0,e-1] \cap [0,x+2]\big) \bs \{ 1,x+1,x+1-e \}, & m= 1.
    \end{cases}
\end{align*}
 
 First assume $m>2$. We have the following cases:
\begin{itemize}
    \item[(a)] If $0 \leq x \leq m-3$, then $|\I| = \big\lvert [0,x+2] \big\rvert = x+3,$ 
    \item[(b)] If $m-2 \leq x \leq e-3$, then $|\I| = \big\lvert[0,x+2] \bs \{m,x+2-m\} \big\rvert =x+1$,
    \item[(c)] If $e-2 \leq x \leq e+m-3$, then $|\I| = \big\lvert[0,e-1]\bs\{m,x+2-m\} \big\rvert =e-2$,
    \item[(d)] If $e+m-2 \leq x \leq 2e-4$, then $|\I| = \big\lvert [0,e-1] \bs \{m\} \big\rvert = e-1$.
\end{itemize}
Let $M$ denote the matrix representation of the first differential operator $\delta_1.$ We now compute $\rk M$ for the above cases. Note that any one-dimensional face is a superset of some zero-dimensional face. In the case of (a) and (b), for any $j<l \in \I$, $$\l -(e+j)-(e+l) =x+2-j-l < e-1.$$
So the only one-dimensional faces in the cases (a) and (b) are $\{e+j,e+l\}$ for some $j<l \in \I$ with $j+l=x+2$. In essence, we pair up the zero-dimensional faces satisfying this condition to get all one-dimensional faces. Thus, 
 \begin{enumerate}
     \item[(a)] for $0 \leq x \leq m-3$, $\rk M = \left\lfloor \frac{x+3}{2}\right\rfloor$, 
     \item[(b)] for $m-2 \leq x \leq e-3$, $\rk M = \left\lfloor \frac{x+1}{2}\right\rfloor$. 
 \end{enumerate}
When $e-2 \leq x \leq 2e-4,$ for $\{e,e+j\},$ $j \in \I \bs \{0\},$ to be a one-dimensional faces of $\D_\l$ we must have 
\begin{align*}
\{e,e+j\} \in \D_\l & \Leftrightarrow  \l -e-(e+j) = x+2-j \in \SM{m}\\
& \Leftrightarrow x+2-j \geq e, \text{ and } \ x+2-j \neq e+m\\
& \Leftrightarrow j \in \big([1,x+2-e] \cap \I \big) \setminus \{ x+2-e-m\}.
\end{align*}
For (c), when $x=e-2$, no such face is possible. When $e-2 < x \leq e + (m-1) -2$, we get $x+2-e$ such faces. Thus, $\rk M \geq x+2-e$.

For (d), if $x=e+m-2,$ then $x+2-e-m=0 $ is not in $[1,x+2-e]$, but $m$ is. If $x>e+m-2,$ then $m$ and $x+2-e-m$ are in $[1,x+2-e]$. If $x+2-e-m=m$, we get $x+1-e$ faces. Otherwise, choosing 
\begin{equation}\label{eq:(e)chooseyform=2}
    y \in [1,m-1] \setminus \{x+2-e-m\}
\end{equation} the face $\{e+y,e+(x+2-e-m)\}$ increases $\rk M$ by one. Thus, $\rk M \geq x+1-e$.

We now look for other one-dimensional faces that add to the rank of $M.$ In particular, the faces of the form $\{e+j,e+l\}$ where $j,l \in [1,e-1] \setminus \{m\}$ with $j \geq x+2-e+1$. This gives us the following $\ll \frac{2e-x-3}{2} \rr$  pairs of $(j,l)$:

\begin{figure}[!h]
    \centering
\begin{tikzpicture}
    \node at (4,3) {$x+2-e+1 < x+2-e+2 < \cdots < e-2 < e-1$};
    \draw (3.5,3.25) -- (3.5,3.50);
   \draw (3.5,3.50) -- (6.5,3.50);
    \draw (6.5,3.50) -- (6.5,3.25);
    \draw (1,2.75) -- (1,2.5);
    \draw (1,2.5) -- (7.75,2.5);
    \draw (7.75,2.5) -- (7.75, 2.75);
\end{tikzpicture}
\end{figure}
In (c), $m,x+2-m$ are in the above sequence but for (d), they are not. Thus,
 \begin{enumerate}
     \item[(c)] for $e-2 \leq x \leq e+m-3$, $\rk M = x+2-e + \ll \frac{2e-x-3}{2} \rr-1 =  \ll \frac{x-3}{2} \rr +1$, 
     \item[(d)] for $e+m-2 \leq x \leq 2e-4$, $\rk M = x+1-e + \ll \frac{2e-x-3}{2} \rr = \ll \frac{x-3}{2} \rr +1$. 
 \end{enumerate}

When $m=2$, the only difference appears in (c) where the choice of the additional one-dimensional face as in \Cref{eq:(e)chooseyform=2} cannot be made. This subtracts one from $\rk M$, increasing $\beta_{1,\lambda}$, and thus, increasing $\beta_1$ for $m=2$ by $1$.

Now suppose $m=1.$ In this case, (a) does not exist, the computations in (b) and (c) remain the same as $\I = \I'$, and so, that leaves the case (d): $ e-1 \leq x \leq 2e-4$. When $x=e$, $x+1-e=1$ so that $\I=\I'$, and again, the calculations do not change. When $x \neq e$, the number of zero-dimensional faces is $|\I'| = \big\lvert [0,e-1]\bs\{1,x+1-e\} \big\rvert =e-2$. The rank also reduces by $1$, i.e., $\rk M = \ll \frac{x-3}{2}\rr,$ as we exclude the one-dimensional face $\{e,e+(x+1-e)\}$. Overall, $\beta_1$ remains the same as in the case of $m > 2.$

We can now apply steps (I)-(III) to get that when $m\neq 2$,
$$ \sum_{\lambda = 2e+2}^{4e-2} \beta_{1, \lambda} = {e-2 \choose 2} -3 +e = {e-1 \choose 2} -1,$$
and when $m=2, \sum_{\lambda = 2e+2}^{4e-2} \beta_{1, \lambda} = {e-1 \choose 2}$.

Note that the special case $\lambda = 2e+2m$ still appears in the intervals above. To get the final $\beta_1$, we only need to subtract the value of $\beta_{1,2e+2m}$ from above and add the actual value of $\beta_{1,2e+2m}$. One can check by looking at the proof of \cite[Lemma 4.4 (1)]{DGS3}, and essentially ignoring the parts including $n$, that the actual value of $\beta_{1,2e+2m}$ is the same as what appears in the interval above. The reason for computing this separately is that the numbers of zero-dimensional faces and the two kinds of one-dimensional faces contributing to $\rk M$ differ from the computations in (a)-(d), despite the final difference being the same. Thus, we have our result.
\end{proof}

For ease of notation, set 
\[ X_e = X_0^2, \quad X_{e+1} =X_0X_1, \quad X_{e+2}=X_1^2. \]

For $m \geq 1$, define
\begin{align*}
B_m =& \left\lbrace 
\begin{aligned}
    &\begin{bmatrix}
   X_0  & X_2 & \cdots & X_{e-2} \\
   X_2  & X_4 & \cdots & X_{e}\\
\end{bmatrix}, & m=1 \\[2mm]
&\begin{bmatrix}
   X_1 & X_3 & \cdots & X_e \\
   X_3 & X_5 & \cdots & X_{e+2}\\
\end{bmatrix}, & m=2 \\[2mm]
&\begin{bmatrix}
   X_0  & X_1 & \cdots & X_{e-4} & X_{e-2} \\
   X_2  & X_3 & \cdots & X_{e-2} & X_{e}\\
\end{bmatrix}, & m=e-1 \\[2mm]
&\begin{bmatrix}
    X_0  & \cdots & X_{m-3} & X_{m-1} & X_{m+1} & \cdots & X_{e-1} \\
    X_2  & \cdots & X_{m-1} & X_{m+1} & X_{m+3} & \cdots & X_{e+1}\\
\end{bmatrix}, & \text{ otherwise. }
\end{aligned} 
\right.\\[5mm]
A_{m} =& \q 
\begin{bmatrix}
    X_0   & \cdots & X_{m-2}& X_{m+1} & \cdots & X_{e-2}&  X_{e-1} \\
    X_1  &  \cdots&  X_{m-1}  & X_{m+2} & \cdots  &X_{e-1}   & X_e\\
\end{bmatrix}.
\end{align*}

The following is a technical lemma used in computing the minimal free resolution of $\IM{1}$ and $\IM{2}$. Although the conclusion is true for all $1 \leq m \leq e-1$, we only make use of it in the aforementioned cases.

\begin{lemma}\label{lem:height of m}
For all $m\ge 1$, ${\bf I}_2(A_m) \subseteq \RM{m}$ is an ideal of height $e-3.$  
\end{lemma}

\begin{proof}
Let $m \notin \{1,e-1\}.$ For any given $u$, consider the numerical semigroup minimally generated by $\{u, u+1, \ldots, 2u-1\}$ and define a map $k[y_0, \ldots, y_{u-1}] \to k[t]$ by $y_i \mapsto t^{u+i}.$ From \cite{GimSenSri}, we know that the defining ideal of the corresponding semigroup ring is the ideal of $2\times 2$ minors of 
$$ Y_u = \begin{bmatrix}
y_0 & \cdots &y_{u-2}& y_{u-1}   \\
y_1 & \cdots&  y_{u-1}  & y_0^2\\
\end{bmatrix}.$$
Hence ${\bf I}_2(Y_u)$ is a prime ideal of height $u-1$. In particular, for a matrix of the above form with $u$ columns, the ideal generated by $2 \times 2$ minors of the matrix is a prime ideal of height $u-1.$

Consider the matrix $$A'_{m} =  
\begin{bmatrix}
    X_0   & \cdots & X_{m-3} & X_{m-2} & X_{m+1} & \cdots & X_{e-2} &  X_{e-1} \\
    X_1  &  \cdots & X_{m-2} & X_{m+1} & X_{m+2} & \cdots  &X_{e-1} & X_0^2\\
\end{bmatrix}.$$ 
Identifying $k[y_0, \ldots, y_{u-1}]$ with $k[X_0, \ldots, X_{m-2}, X_{m+1}, \ldots, X_{e-1}]=\RM{m-1,m}$ in the natural order, as $A'_m$ has $e-2$ columns, we get ${\bf I}_2(A'_m)$ is a (prime) ideal of height $e-3$ in  $\RM{m-1,m}$ and so in $\RM{m}$. Further, ${\bf I}_2(A'_m)+(X_{m-1}-X_{m+1})$ has height $e-2$ in $\RM{m}.$ As ${\bf I}_2(A_m)+(X_{m-1}-X_{m+1}) = {\bf I}_2(A'_m)+(X_{m-1}-X_{m+1}),$ we get height of ${\bf I}_2(A_{m})$ is $e-3$. 
Now, for $m \in \{1,e-1\},$ since $A_m$ is a generic matrix of size $2 \times (e-2),$ we get ${\bf I}_2(A_m)$ is a prime ideal of height $e-3.$
\end{proof}

\begin{theorem}  \label{thm:defining ideal m}
    For all $m \geq 1,$ the defining ideal $\IM{m} = {\bf I}_2(A_m) + {\bf I}_2(B_m),$ and is minimally generated by all the ${e-2\choose 2}$ minors of size $2$ in $A_m$ and the size $2$ minors of $B_m$ containing the column $\begin{bmatrix} X_{m-1}\\ X_{m+1}\end{bmatrix}$.  
\end{theorem}

\begin{proof}
    Recall that $\IM{m}$ is the kernel of the map from $\RM{m}$ to $k[\SM{m}]$ that maps $x_i$ to $t^{e+i}.$ Since the difference between the subscripts in the second and the first row is constant, $1$ in the case of the matrix $A_m$ and $2$ for the matrix $B_m$, the $2\times 2$ minors of $A_m$ and $B_m$ are automatically in the kernel. Now the $2\times 2$ minors of $A_m$ are all minimal generators of $\IM{m}$ as they are minimal generators of ${\bf I}_2(A_m).$ 
    
    However, none of the minors of $B_m$ containing the column of $\begin{bmatrix} X_{m-1}\\ X_{m+1}\end{bmatrix}$ are in ${\bf I}_2(A_m)$ and they are all part of a minimal generating system of ${\bf I}_2(B_m)$ and are minimal in $\IM{m}$. Indeed, it is clear that $X_2X_{m-1}-X_0X_{m+1}$ is not in ${\bf I}_2(A_m)$ because there is no binomial with a monomial of the form $X_0X_{m+1}$ in ${\bf I}_2(A_m).$ If we are to get $X_{m-1}X_{j+2}-X_{m+1}X_j$, $1\le j \le m-3$, from either ${\bf I}_2(A_m)$ or any minor from ${\bf I}_2(B_m)$ preceeding it, we cannot reach $X_{m+1}X_j$ at all as all these minors will involve variables of the form $X_i$, $i<m,$ from $A_m$ and $X_{m+1}X_i, i<j$. Similarly, for $j \geq 1,$ considering $X_{m-1}X_{m+j+2}-X_{m+1}X_{m+j}$, it is clear that $j=1$ is minimal since that is the only binomial where the term $X_{m+1}^2$ appears. Now for $j>1$, any term from $A_m$ involving $X_{m-1}X_{m+j+2}$ will involve terms of the form $X_{m-i}X_{m+k}$ and the ones already accounted for with $X_{m+k}, k<j+2$, will involve $X_{m+1}X_{m+k}$ $k< j$.  Thus we will never reach $X_{m+1}X_{m+j}$. 
\end{proof}

\subsection{Minimal Free Resolutions.} We construct minimal free resolutions for the cases $m=1$ and $m=2$ by identifying them as mapping cones associated to certain Eagon-Northcott complexes. As a consequence, we get formulas for their Betti sequence.

We set the following notation. Let $$F_m = \sum_{i=0, i \neq m-1,m}^{e-1} \RM{m}f_i, \quad \text{and} \q G= \RM{m}g_1 \oplus \RM{m}g_2$$ be free $\RM{m}$-modules with bases $\{f_i \mid i \in [0,e-1] \bs \{m-1,m\}\}$ and $\{g_1, g_2\}$, respectively. Let $\phi : F_m \to G$ be given by the matrix $A_m$ so that $\phi(f_i)=X_ig_1 +X_{i+1}g_2$ for all $i \in [0,e-1] \bs \{m-1,m\}$. Let $\E$ denote the Eagon-Northcott complex of $\phi$ with the $i$-th module $E_i= \wedge^{1+i}F_m \otimes D_{i-1}G^\ast$ for all $i \in [1,e-3]$.

\begin{theorem} [Resolution of $\IM{1}$] \label{thm:res of I1}
Let $\E^\ast$ be the dual of $\E$, and define $\psi : D_{e-4} G \to \RM{1}$ by $$\psi \left( g_1^j g_2^{e-4-j} \right) = (-1)^j (X_0X_{j+4}-X_2X_{j+2}), \q j \in [0,e-4].$$
This map $\psi$ extends to a map of complexes $\E^\ast \to \E$ whose mapping cone is a resolution of $\IM{1}$.
\end{theorem}

\begin{proof}
First note that $A_1$ is a $2 \times e-2$ generic matrix and by \Cref{lem:height of m}, ${\bf I}_2(A_1)$ has height $e-3,$ implying that both $\E$ and its dual $\E^\ast$ are exact. Consider the diagram 

\[\begin{tikzcd}
	{\mathbf E: 0} & {\overset{e-2}{\wedge} F_1\otimes D_{e-4}G^\ast} & \cdots & {\overset{2}{\wedge}F_1} & {\RM{1}} \\
	{\mathbf E^\ast: 0} & {\RM{1}^\ast} & \cdots & {\overset{e-3}{\wedge} F_1^\ast \otimes D_{e-5}G} & {\overset{e-2}{\wedge} F_1^\ast \otimes D_{e-4}G}
	\arrow[from=1-1, to=1-2]
	\arrow["{\delta_{e-3}}", from=1-2, to=1-3]
	\arrow["{\delta_2}", from=1-3, to=1-4]
	\arrow["{\delta_1}", from=1-4, to=1-5]
	\arrow[from=2-1, to=2-2]
	\arrow["{\delta_1^\ast}", from=2-2, to=2-3]
	\arrow[from=2-3, to=2-4]
    \arrow["{\psi_{e-3}}", dashed, from=2-2, to=1-2]
	\arrow["{\psi_1}"', dashed, from=2-4, to=1-4]
	\arrow["{\delta^\ast_{e-3}}", from=2-4, to=2-5]
	\arrow["\psi"', from=2-5, to=1-5]
\end{tikzcd}\]
where we make the identification $\wedge^i F_1^\ast \cong \wedge^{e-2-i} F_1$ by the canonical isomorphism with appropriate signs. To show that $\psi$ extends to a map of complexes, it suffices to prove that the image of the composition $\psi \circ \delta^\ast_{e-3}$ is contained in ${\bf I}_2(A_1)$. Indeed, for some $i\in [2,e-1],$ and $j \in [0,e-5]$,
\begin{align*}
    (\psi \circ \delta_{e-3}^\ast) (f_i \otimes g_1^j g_2^{e-5-j}) &= \psi(X_i g_1^{j+1}g_2^{e-5-j} + X_{i+1}g_1^jg_2^{e-4-j}) \\
    &= (-1)^{j+1}X_j(X_0 X_{j+5}-X_2X_{j+3}) + (-1)^{j}X_{i+1}(X_0 X_{j+4}-X_2X_{j+2}) \\
    &= (-1)^j \left[X_0(X_{i+1}X_{j+4}-X_iX_{j+5})+X_2(X_iX_{j+3}-X_{i+1}X_{j+2}) \right] \\
    & \in {\bf I}_2(A_1).
\end{align*}
Hence, $\psi$ extends to a map of complexes $\E^\ast \to \E$, and thus induces the map $\psi_1$ defined by
\begin{equation}\label{eq:psi1 I2}
    \psi_1(f_i \otimes g_1^j g_2^{e-5-j}) = X_0(f_i \wedge f_{j+4}) + X_2(f_i \wedge f_{j+2}) \in \wedge^2 F_1.
\end{equation}

The mapping cone of this map of complexes is exact and is a resolution of ${\bf I}_2(A_1)+{\bf I}_2(B_1)=\IM{1}$.  
\end{proof}

\begin{corollary}\label{Betti1}
    The resolution in \Cref{thm:res of I1} is minimal. Hence, for $1 \leq t \leq e-3,$ the Betti numbers of $k[\SM{1}]$ are $$\beta_t = {{et}\over {(e-t-1)}}{e-2 \choose t+1}, $$
    and $\beta_{e-2} = 1.$
\end{corollary}

\begin{proof}
    Let $\psi: \mathbf E^\ast \to \mathbf E$ be the map of complexes as in \Cref{thm:res of I1}. With the bases of $F_{1}$ and $G$ as in the notation, the non-zero entries in the matrix of $\psi_1$ are $X_0$ and $X_2$, both of degree $1$ in the standard grading.  The entries in $\delta_i,$ $i\ge 2$ in $E$ and $E^*$ have $X_i, i\ge 2$ and $X_0^2$.  Thus, $\psi_{i-1} \circ \delta_{e-2-i}^* $ will have coefficients of degree $2$ when they do not involve $X_0$ and of degree $3$ when it involves $X_0$.  Hence, the corresponding $\psi_i$ must have degree one terms.  It follows that the non-zero entries in the matrices of $\psi_i, i\ge 1$ must all be of degree $1$, and hence, the resolution is minimal. 

For $t \in [1,e-3],$ the $t$-th free module in the resolution is $E_t\oplus E^\ast_{e-2-t}$. Hence the Betti numbers are $$\beta_t = t{e-2\choose t+1}+(e-2-t){e-2\choose e-1-t} = {{et}\over {(e-t-1)}}{e-2 \choose t+1}.$$
Since $\RM{1}^*$ is the free module at the $(e-2)$-nd spot in the mapping cone, we get $\beta_{e-2}=1.$ 
\end{proof}

\begin{remark}   
Note that when $m=e-1,$ $\SM{e-1} = \la e,e+1,\ldots,2e-2 \ra$ is a numerical semigroup generated by an arithmetic sequence, with multiplicity $e,$ width $e-2,$ and embedding dimension $e-1.$ Minimal free resolution of such semigroups is computed in \cite{GimSenSri}. In particular, with the setup as described below, one can obtain a minimal free resolution of $\IM{e-1}$ as in the \Cref{thm:res of I1} by making the replacement $$\psi(g_1^j g_2^{e-4-j}) = (-1)^j(X_eX_{j} -X_{e-2}X_{j+2}), \q j \in [0,e-4].$$

The map $\delta^\ast_{e-3}$ remains the same, and one can check that $$\psi \circ \delta_{e-3}^\ast (f_i \otimes g_1^j g_2^{e-5-j}) = (-1)^{j}[X_{e-2}(X_iX_{j+3}-X_{i+1}X_{j+2})+X_{e}(X_{i+1}X_j - X_iX_{j+1})]\in {\bf I}_2(A_{e-1}).$$ 
\end{remark}

We now compute a minimal free resolution of $\IM{2}.$ Recall that when $m \in \{1,e-1\},$ $\SM{m}$ is a symmetric numerical semigroup, implying that the Betti numbers are symmetric (that is, $\beta_t = \beta_{e-2-t}$ for $t \in [0,e-2]$). Since this is not the case when $m=2,$ the construction of the resolution for this case doesn't follow in parallel. 

\begin{theorem}[Resolution of $\IM{2}$] \label{thm:res of Im} 
Let $\mathbf E'$ be the Buchsbaum-Rim complex associated to the $2\times (e-2)$ matrix 
$$
\begin{bmatrix}
    X_0^2   & X_3  & \cdots & X_{e-2}&  X_{e-1} \\
    X_0X_1  & X_4  & \cdots  &X_{e-1}   & X_0^2= X_e\\
\end{bmatrix},$$
and $\mathbf E$ be the Eagon-Northcott complex associated to $A_2.$
Recall that $F_2 = \sum_{i=0, i \neq 1,2}^{e-1} \RM{2}f_i$ and $G= \RM{2}g_1 \oplus \RM{2}g_2$ are free $\RM{2}$-modules of rank $e-2$ and $2$ respectively. Then the $i$-th module of $\mathbf E'$ is $E_i'=\wedge^{i+1} F_2 \otimes D_iG^\ast$, $i \in [0,e-3]$, so that $E_0'=F_2 \cong \RM{2}^{e-2}$. Define $\psi: F_2 \to \RM{2}$ by 
$$\psi(f_j)=\begin{cases}
    X_{3}X_{j} - X_{1}X_{j+2}, &3\le j \le e-1  \\
    X_0^2X_3-X_1^3, & j=0.
\end{cases}$$
Then $\psi$ extends to a map of complexes $\E' \to \E$ whose mapping cone is a resolution of $\IM{2}$.
\end{theorem}
\begin{proof}
    From \Cref{lem:height of m}, we know that the height of ${\bf I}_2(A_2)$ is $e-3.$ Using similar arguments, we can show that the height of the size 2 minors of the above matrix is $e-3.$ Thus, $\E$ and $\E'$ are exact, and we have the following diagram with exact rows:

\[\begin{tikzcd}
	{\mathbf E: 0} & {\overset{e-2}{\wedge} F_2\otimes D_{e-4}G^\ast} & \cdots & {\overset{2}{\wedge}F_2} & {\RM{2}} \\
	\E': 0 & {\overset{e-2}{\wedge} F_2\otimes D_{e-3}G^\ast} & \cdots & {\overset{2}{\wedge}F_2} \otimes D_1G^\ast & F_2
	\arrow[from=1-1, to=1-2]
	\arrow["{\delta_{e-3}}", from=1-2, to=1-3]
	\arrow["{\delta_2}", from=1-3, to=1-4]
	\arrow["{\delta_1}", from=1-4, to=1-5]
	\arrow[from=2-1, to=2-2]
	\arrow["{\delta_{e-3}'}", from=2-2, to=2-3]
	\arrow["{\delta_2'}", from=2-3, to=2-4]
	\arrow["{\psi_1}"', dashed, from=2-4, to=1-4]
    \arrow["{\psi_{e-3}}", dashed, from=2-2, to=1-2]
	\arrow["{\delta_1'}", from=2-4, to=2-5]
	\arrow["\psi"', from=2-5, to=1-5]
\end{tikzcd}\]

    To prove that $\psi$ extends to a map of complexes $\E' \to \E$, we will establish that $\psi \circ \delta_1' (f_i \wedge f_j \otimes g_l) \in {\bf I}_2(A_2)$ for all $i,j \in \{0\} \cup [3,e-1],$ and $l \in \{1,2\}$.

    If $i=j$, then $\psi \circ \delta'=0 \in {\bf I}_2(A_2)$. If $i=0,$ then for all $j \in [3,e-1]$,
    \begin{align*}
        \psi \circ \delta_1' (f_0 \wedge f_j \otimes g_1) &= \psi(X_{0}^2f_j-X_{j}f_0) \\
        &= X_{0}^2(X_{3}X_{j} -X_{1}X_{j+2})-X_{j}(X_0^2X_{3}- X_{1}^3) \\
        &= X_1(X_1^2X_{j}-X_0^2X_{j+2}) \in {\bf I}_2(A_2)
\end{align*}
with
\begin{equation}\label{eq:psi1 f0 g1 I2}
    \psi_1(f_0 \wedge f_j \otimes g_1)= \begin{cases}
        X_0X_{1}(f_0 \wedge f_{j+1}) + X_1^2(f_0 \wedge f_j), & j \leq e-2 \\
        X_1^2 (f_0 \wedge f_{e-1}), & j=e-1
    \end{cases}
\end{equation}
and \begin{align*}
        \psi \circ \delta_1' (f_0 \wedge f_j \otimes g_2) &= \psi(X_0X_1f_j-X_{j+1}f_0) \\
        &= X_{0}X_1(X_{3}X_{j} -X_{1}X_{j+2})-X_{j+1}(X_0^2X_{3}- X_{1}^3) \\
        &= X_0X_3(X_1X_{j}-X_0X_{j+1}) + X_1^2(X_1X_{j+1}-X_0X_{j+2}) \in {\bf I}_2(A_2)
\end{align*}
with
\begin{equation}\label{eq:psi1 f0 g2 I2}
    \psi_1(f_0 \wedge f_j \otimes g_2)= \begin{cases}
        X_0X_3(f_0 \wedge f_{j}) + X_1^2(f_0 \wedge f_{j+1}), & j \leq e-2 \\
        X_0X_3 (f_0 \wedge f_{e-1}), & j=e-1.
    \end{cases}
\end{equation}

Suppose $3\leq i<j \leq e-1$. Then 
\begin{align*}
        \psi \circ \delta_1' (f_i \wedge f_j \otimes g_1) &= \psi(X_{i}f_j-X_{j}f_i) \\
        &= X_{i}(X_{3}X_{j} -X_{1}X_{j+2})-X_{j}(X_{3}X_{i} - X_{1}X_{i+2}) \\
        &= X_1(X_{i+2}X_j-X_iX_{j+2})\in {\bf I}_2(A_2)
\end{align*}
with 
\begin{equation}\label{eq:psi1 g1 I2}
    \psi_1(f_i \wedge f_j \otimes g_1)= \begin{cases}
        X_{1}(f_i \wedge f_{j+1} + f_{i+1}\wedge f_j), & i<j \leq e-2 \\
        X_1(f_{i+1} \wedge f_{j}) + X_0X_1 (f_0 \wedge f_i), & i \leq e-3, j=e-1 \\
        X_0X_1(f_0 \wedge f_i), & i=e-2, j=e-1.
    \end{cases}
\end{equation}

Similarly, 
    \begin{align*}
        \psi \circ \delta_1' (f_i \wedge f_j \otimes g_2) &= \psi(X_{i+1}f_j-X_{j+1}f_i) \\
        &= X_{i+1}(X_{3}X_{j} -X_{1}X_{j+2})-X_{j+1}(X_{3}X_{i} - X_{1}X_{i+2}) \\
        &= X_3(X_{i+1}X_{j}-X_{i}X_{j+1})+X_1(X_{i+2}X_{j+1}-X_{i+1}X_{j+2}) \in {\bf I}_2(A_2) 
\end{align*}
and 
\begin{equation}\label{eq:psi1 g2 I2}
    \psi_1(f_i \wedge f_j \otimes g_2)= X_3(f_{i}\wedge f_{j}) + X_1(f_{i+1} \wedge f_{j+1}).
\end{equation}

Hence, we can deduce that $\psi$ extends to a map of complexes $\E' \to \E$. The mapping cone of this map of complexes is exact and is a resolution of $\IM{1}$.  
\end{proof}

\begin{corollary}\label{Betti2}
    The resolution in \Cref{thm:res of Im} is minimal. Hence, the Betti numbers of $k[\SM{2}]$ are $$\beta_t = t {e-1 \choose t+1}, \quad 1 \leq t \leq e-2.$$
\end{corollary}
\begin{proof}
 Let $\psi: \mathbf E'\to \mathbf E$ be the map of complexes as in the \Cref{thm:res of Im}.  If we choose bases for $F_{2}$ and $G$ as above, the non-zero  entries in the matrix of $\psi_1$ are of degree $1$ or $2$ in the standard grading as can be seen from equations \ref{eq:psi1 f0 g1 I2}, \ref{eq:psi1 f0 g2 I2}, \ref{eq:psi1 g1 I2}, and \ref{eq:psi1 g2 I2}. The only entries of degree $2$ are of the form $X_0X_1, X_0X_3,$ or $X_1^2$ and these monomials cannot be a part of any binomial in ${\bf I}_2(A_2)$.  Now, the nonzero entries in the matrices of the resolutions $\mathbf E$ and $\mathbf E'$ from $E_2$ onwards are again of degree $1$.  It follows by induction that the nonzero entries in the matrices of $\psi_i, i\ge 1$ are either degree $1$ monomials, or degree $2$ monomials that do not appear in any binomials in ${\bf I}_2(A_2)$. Hence, the resolution is minimal. 

The $t$-th free module in the resolution is $E_t\oplus E'_{t-1}.$ Thus, the Betti numbers are 
\[\beta_t = t{e-2 \choose t+1}+ t {e-2 \choose t}= t{e-1 \choose t+1}. \qedhere\]  
\end{proof}

\section{Structure of $\SM{m,n}$}

In the following sections, we study Sally type numerical semigroups of the form $\SM{m,n}$ and prove results analogous to Section \ref{sec:drop1}.

\subsection{The Numerical Semigroup $\SM{m,n}$}

For $1 \leq m < n \leq e-1$, $\SM{m,n}$ is the numerical semigroup generated by $[e, 2e-1] \bs \{e+m,e+n\}$. We begin by computing its Frobenius number and Cohen-Macaulay type.
\begin{remark}  \label{rem:e at least 6}
    We assume that $e\geq 6$ since if $e=4$ the semigroup is symmetric and if $e=5$ we have a space monomial curve and the properties of interest are known in these cases \cite{Herzog}.
\end{remark}

\begin{theorem}
The Frobenius number, $F(\SM{m,n})$, of $\SM{m,n}$ is
    $$F(\SM{m,n})= \begin{cases}
        2e+3, & (m,n) = (2,3) \hskip33pt (i) \\
        2e+n, & m=1, \ n\in \{2,3\} \hskip18pt (ii)\\
        2e+1, & m=1, \ n \geq 4  \hskip36pt  (iii)\\
        e+n, & \text{ otherwise. } \hskip44pt  (iv)
        \end{cases} $$
\end{theorem}
\begin{proof}
    Parts (i) and (iv) for $m \ne 1$ and $n \leq e-2$ are proven in \cite[Proposition 2.2]{DGS3}. One can use similar arguments to prove (iv) for $n=e-1$ and $m \neq 1.$

    Note that when $m=1$, $2e+1 \not\in \SM{1,n}$. Moreover, when $n\geq 4,$ for any integer $y \in [1,e],$ $2e+1+y = e+(e+1+y) \in \SM{1,n}$ unless $y+1=n$ or $y=e.$ When $y+1=n$, 
    $$2e+n = (e+2) + (e+n-2) \in \SM{1,n}$$
    and when $y=e,$ 
    \[ 2e+1+y = 3e+1 = (e+2) + (2e-1) = (e+3) + (2e-2) \in \SM{1,n}. \]
    However, when $n\in \{2,3\}$, $2e+n \not\in \SM{1,n}$. Again, let $y \in [1,e]$ be an integer. If $n=2$, then $$2e+n+y = \begin{cases}
        e+ (e+y+2), &  y \in [1, e-2] \\
        (e+3) + (e+e-2), & y=e-1 \\
        (e+3) + (e+e-1), & y=e
    \end{cases}$$ all of which lie in $\SM{1,2}$. If $n=3,$ then $$2e+n+y = \begin{cases}
        (e+2) + (e+y+1), &  y \in [1,e-1] \bs \{2\} \\
        e + (e+5), & y = 2 \\
        (e+4) + (e+e-1), & y=e 
    \end{cases}$$ all lying in $\SM{1,3}$.
\end{proof}

\begin{theorem} \label{thm:type m,n}
For $1 \leq m < n \leq e-1$, the Cohen-Macaulay type of $\SM{m,n}$ is,
$$t(\SM{1,n}) = \begin{cases}
        2, & n\in \{2,3\} \bigcup \left[ \frac{e+1}{2}, e-1 \right] \bs \lbrace \frac{e+2}{2} \rbrace \\
        3, &  n \in \left[ 4, \frac{e+1}{2} \right) \bigcup \lbrace \frac{e+2}{2} \rbrace 
    \end{cases}, \q t(\SM{m\geq 2,n}) = \begin{cases}
        1, & (m,n)=(2,3) \\
        e-n+1, & n \leq 2m, n \neq 3 \\
        e-n+2, & n>2m.
    \end{cases}$$
In particular (and in light of Remark \ref{rem:e at least 6}), $\SM{m,n}$ is symmetric if and only if $(m,n)=(2,3)$ and $\SM{m,n}$ is almost symmetric if and only if $m \geq 2,$ and $n>2m.$ 
\end{theorem}

\begin{proof}
Recall that $t(\SM{m,n}) = |\PF(\SM{m,n})|$. 

\textbf{Case 1}: Suppose $m=1.$ Set $V = [1,e-1] \cup \{e+1\}.$ Note that $V \subseteq G(\SM{1,n})$ for any $n.$ We claim that for any $x \in V \bs \{e-n+1\},$ $x$ is not a pseudo-Frobenius number of $\SM{1,n}$. For $x \in [1,e-1] \setminus \{e-n+1\},$ there exists $j \in [1,e-1]\setminus \{n-1\}$ with $x=e-j$ such that $e+(j+1)\in \SM{1,n}$ but $x+e+(j+1)=2e+1 \notin \SM{1,n}.$ Therefore, $x \notin \PF(\SM{1,n}).$ It is easy to check that for $x = e+1,$ as $e+x = 2e+1 \notin \SM{1,n},$ we get $x \notin \PF(\SM{1,n}).$ 

Additionally, let $n \in \{2,3\}.$ Then $G({\SM{1,n}})= V \cup \{e+n,2e+1,2e+n\}$ as $2e+n$ is the Frobenius number of $\SM{1,n}$ in this case. We know that no element of $V \bs \{e-n+1\}$ is a pseudo-Frobenius element. For $x=e-n+1,$ we have $e+(2n+1) \in \SM{1,n}$ (as $n=2,3$ and $e\geq 6$) whereas $x+e+(2n-1) \notin \SM{1,n},$ implying that $e-n+1 \notin \PF(\SM{1,n}).$ Moreover, as $e+(e+n) = 2e+n \notin \SM{1,n},$ we get $e+n \notin \PF(\SM{1,n}).$ The gap $2e+1 \in \PF(\SM{1,n})$ since  for any nonzero $y \in \SM{1,n}$, 
$$2e+1+y>2e+n$$ and as $2e+n = F(\SM{1,n}),$ it is clearly is a pseudo-Frobenius number. Thus, in this case, $\PF(\SM{1,n}) = \{2e+1,2e+n\}$ implying that $t(\SM{1,n}) = 2.$

Suppose $m=1$, and $n \neq 2,3.$ Then $G({\SM{1,n}}) = V \cup \{e+n,2e+1\}.$ As above, we know that no element of $V \bs \{e-n+1\}$ is a pseudo-Frobenius element. On the other hand, for any nonzero element $y\in \SM{1,n},$
$$y+(e+n)>2e+1=F(\SM{1,n})$$ 
so $e+n \in \PF(\SM{1,n})$. Since $2e+1$ is the largest gap of $\SM{1,n},$ it is a pseudo-Frobenius element. We now claim that $e-n+1 \in \PF(\SM{1,n})$ if and only if $n<(e+1)/2$ or $n=(e+2)/2.$ For $n= (e+2)/2,$ $x=e-n+1=n-1$. Hence, $e+x \in \SM{1,n}$ and for any $j \in [2,e-1] \setminus \{n\},$ $x+(e+j)=e+n+j-1>e+n.$ Furthermore, $x+(e+j)\neq 2e+1$ as $j\neq n$. So, $e-n+1$ is a pseudo-Frobenius number if $n= (e+2)/2.$ Now assume that  $n \neq (e+2)/2.$ If $n<(e+1)/2,$ then for $x=e-n+1$ and for any $j\in [0,e-1] \setminus \{1,n\}$,
$$x+(e+j)=2e+j-n+1>e+n+j> e+n$$
and $x+(e+j) \neq 2e+1$ as $j\neq n$. Thus $x = e-n+1 \in \PF(\SM{1,n}).$

For $n \ge (e+1)/2,$ and $x=e-n+1,$ pick $j=n-x=2n-e-1\geq 0$. Note that $j<e-1$  and $j \notin \{1,n\}$ as $n \neq (e+2)/2$ and as $2n-e-1=n$ implies $n=e+1$ which is a contradiction. Then $x+e+j=e+n\notin \SM{1,n}$ implying that $e-n+1 \notin \PF(\SM{1,n})$ in this case. Thus when $n\neq 2,3$,
$$\PF(\SM{1,n}) = \begin{cases}
 \{e+n,2e+1\}, & n \geq (e+1)/2 \text{ and } n \neq (e+2)/2 \\
 \{e+n,2e+1,e-n+1\}, & n<(e+1)/2 \text{ or } n = (e+2)/2,
\end{cases}$$
giving us the required Cohen-Macaulay type of $\SM{1,n}.$

\textbf{Case 2}: Suppose $m \ge 2.$ When $(m,n) = (2,3)$, from \cite[Theorem 2.3]{DGS3} we know that $\SM{2,3}$ is symmetric and thus $\PF(\SM{2,3}) = \{F(\SM{2,3})\}$ (\cite[Corollary 4.11]{Rosales2009}) implying that $t(\SM{2,3})=1.$ Now assume $(m,n)\neq (2,3)$.

When $2\leq m<n\leq e-1$, since the Frobenius number $F(\SM{m,n})=e+n,$ it follows that the gap set, $G(\SM{m,n}) = [1,e-1] \cup \{e+m,e+n\}$. We claim that in this case, the set of pseudo-Frobenius elements,
$$\PF(\SM{m,n}) = \begin{cases}
G(\SM{m,n})\setminus [1,n], & n \leq 2m \\
\big( G(\SM{m,n})\setminus [1,n] \big) \cup \{n-m\}, & n> 2m.
\end{cases}$$
To see this, if $x\in G(\SM{m,n})\setminus [1,n]$ then for any $j\in [0,e-1]\setminus \{m,n\}$,
$$x+(e+j)\geq e+n+1>F(\SM{m,n}).$$ So these gaps are all pseudo-Frobenius numbers. If $x$ is a gap in $ [1,n]\setminus\{n-m\}$, then there exists $j\in [0,n-1]\setminus \{m\}$ such that $x=n-j.$ Note that for any $j \in [0,n-1] \setminus \{m\},$ $e+j\in \SM{m,n}$ but  
$$x+(e+j) = (n-j)+(e+j)=e+n\notin \SM{m,n}$$
implying that any $x \in [1,n]\setminus\{n-m\}$ is not a pseudo-Frobenius number. Finally, consider $x=n-m$. If $n\leq 2m$, pick $j=m-x=2m-n \geq 0$. Note that $j < e-1$, and $j \not\in \{m,n\}$ as $2m-n \in \{m,n\}$ implies $m=n$, a contradiction. Thus $j \in [0,e-1] \bs \{m,n\}$, and $e+j +x = e+m \not\in \SM{m,n}$, which implies $x \not\in \PF(\SM{m,n})$.

If $n>2m$, then for $x=n-m$ and $j \in [0,e-1] \bs \{m,n\}$, $$x+e+j = e+j + n-m > e+2m-m = e+m,$$
and $e+j+x \neq e+n$ as $j \neq m$. In other words, $e+j + n-m \in \SM{m,n}$ for all $j \in [0,e-1] \bs \{m,n\}$, which implies $n-m \in \PF(\SM{m,n})$. Thus, 
\[ t(\SM{m\geq 2,n}) = \begin{cases}
        1, & (m,n)=(2,3) \\
        e-n+1, & n \leq 2m, n \neq 3 \\
        e-n+2, & n>2m.  
    \end{cases} \qedhere \]
\end{proof}

\subsection{The Defining Ideal $\IM{m,n}$.}
Let $\JM{m,n} = [0,e-1] \bs \{m,n\}$ and $\RM{m,n}=k[X_i \mid i\in \JM{m,n}]$ be a polynomial ring over a field $k$ with weighted degrees, $\deg X_i = e+i,$ for all $i \in \JM{m,n}.$ Then the semigroup ring $k[\SM{m,n}] \cong \RM{m,n}/\IM{m,n}$ has  embedding dimension $e-2,$ and multiplicity $e$, where $\IM{m,n}$ is the defining ideal of $\SM{m,n}$. We end this section with a technical result analogous to \Cref{lem:height of m}. Note that, as in \Cref{lem:height of m}, the conclusion is true for all $2 \leq m \leq e-2$, but we only utilize it to compute the resolutions of $\IM{2,3}$ and $\IM{3,4}$.

\noindent Let  
$$A_{m,m+1} =  \begin{bmatrix}
        X_0   & \cdots &X_{m-2}& X_{m+2} & \cdots & X_{e-2}&  X_{e-1} \\
        X_1  &  \cdots&  X_{m-1}  & X_{m+3} & \cdots  &X_{e-1}   & X_0^2\\
\end{bmatrix}.$$

\begin{lemma}\label{lem:height of m and m+1}
For all $m\ge 2,$ ${\bf I}_2(A_{m,m+1})\in \RM{m,m+1}$ is an ideal of height $e-4.$ 
\end{lemma}

\begin{proof}
Recall the ($2 \times u$) matrix $Y_u$ from Lemma \ref{lem:height of m}, and that ${\bf I}_2(Y_u)$ is a prime ideal of height $u-1.$ Consider 
$$A'_{m,m+1} =  
\begin{bmatrix}
    X_0   & \cdots & X_{m-3} & X_{m-2}& X_{m+2}  & \cdots & X_{e-2}&  X_{e-1} \\
    X_1  &  \cdots & X_{m-2} &  X_{m+2}  & X_{m+3} & \cdots  &X_{e-1}   & X_0^2\\
\end{bmatrix}.$$
Identifying $k[y_0, \ldots, y_{u-1}]$ with $\RM{m-1,m,m+1}$ in the natural order, as $A'_{m,m+1}$ has $u=e-3$ columns, we get ${\bf I}_2(A'_{m,m+1})$ is a (prime) ideal of height $e-4$ in $\RM{m-1,m,m+1}$ and so in $\RM{m,m+1}.$ Further, ${\bf I}_2(A'_{m,m+1})+(X_{m-1}-X_{m+2})$ has height $e-3$. As ${\bf I}_2(A_{m,m+1})+(X_{m-1}-X_{m+2}) = {\bf I}_2(A'_{m,m+1})+(X_{m-1}-X_{m+2}),$ we have height of ${\bf I}_2(A_{m,m+1})$ is $ e-4$.
\end{proof}

\section{Symmetric Sally $\SM{2,3}$}\label{Gorenstein}
In this section, we focus on the Gorenstein case, i.e., we study the Sally semigroup $\SM{2,3}$ by computing its defining ideal, a minimal free resolution, and the Betti sequence.

For ease of notation, set 
\[ X_e = X_0^2, \quad X_{e+1} =X_0X_1, \quad X_{e+2}=X_1^2, \]
and let $A_{2,3}$ and $B_{2,3}$ be the matrices

$$A_{2,3} = \begin{bmatrix}
        X_0   & X_4 & X_5 & \cdots &  X_{e-2} & X_{e-1} \\
        X_1   & X_5 & X_6  & \cdots    & X_{e-1} & X_e
        \end{bmatrix}, \q B_{2,3} = \begin{bmatrix}
          X_1   & X_4 & X_{5} & \cdots & X_{e-2}&X_{e-1}\\
   X_4  & X_{7} & X_{8}  & \cdots & X_{e+1} & X_{e+2}\\ 
    \end{bmatrix}.$$

\begin{theorem}
The defining ideal $\IM{2,3}$ of $\SM{2,3}$ is  $$\IM{2,3}= {\bf I}_2(A_{2,3})+{\bf I}_2(B_{2,3}).$$
In particular, the minimal generators of $\IM{2,3}$ are the ${e-3\choose 2}$ size $2$ minors of $A_{2,3}$ and the $e-4$ size $2$ minors of $B_{2,3}$ that involve the first column. 
\end{theorem}  
\begin{proof}
 This is a direct consequence of \cite[Theorem 5.2]{DGS3} in which the authors compute a minimal generating set for $\IM{2,3}$.
\end{proof}
   
   \begin{theorem} \label{thm:res of I23}
Let $\displaystyle{F_{2,3} = \sum_{i=0,i \neq 1,2,3}^{e-1} \RM{2,3}f_i}$, $G = \RM{2,3}g_1\oplus \RM{2,3}g_2$ be free $\RM{2,3}$-modules, and $\phi: F_{2,3} \to G$ be given by the matrix $A_{2,3}$ so that $\phi(f_i) = X_i g_1+X_{i+1}g_2,$ for $i \in \{0\} \cup [4, e-1].$ Let $\mathbf{E}$ denote the Eagon-Northcott complex of $\phi$ and $\mathbf{E^*}$ denote its dual. Construct a map $\psi : D_{e-5}G \to \RM{2,3}$ as follows:
\begin{equation*} \psi( g_1^jg_2^{e-5-j}) = (-1)^j (X_1X_{j+7}-X_4X_{j+4}), \q 0\le j\le e-5.
\end{equation*}
Then the map $\psi$ extends to a map of complexes  $\psi: \mathbf{E^*}\to \mathbf{E}$ whose mapping cone is the resolution of $\IM{2,3}$.

\end{theorem}
\begin{proof} Consider the following diagram
\[\begin{tikzcd}
	{\mathbf E: 0} & {\overset{e-3}{\wedge} F_{2,3}\otimes D_{e-5}G^\ast} & \cdots & {\overset{2}{\wedge}F_{2,3}} & {\RM{2,3}} \\
	{\mathbf E^\ast: 0} & {\RM{2,3}^\ast} & \cdots & {\overset{e-4}{\wedge} F_{2,3}^\ast \otimes D_{e-6}G} & {\overset{e-3}{\wedge} F_{2,3}^\ast \otimes D_{e-5}G}
	\arrow[from=1-1, to=1-2]
	\arrow["{\delta_{e-4}}", from=1-2, to=1-3]
	\arrow["{\delta_2}", from=1-3, to=1-4]
	\arrow["{\delta_1}", from=1-4, to=1-5]
	\arrow[from=2-1, to=2-2]
	\arrow["{\delta_1^\ast}", from=2-2, to=2-3]
	\arrow[from=2-3, to=2-4]
    \arrow[dashed, from=2-2, to=1-2]
	\arrow["{\psi_1}"', dashed, from=2-4, to=1-4]
	\arrow["{\delta^\ast_{e-4}}", from=2-4, to=2-5]
	\arrow["\psi"', from=2-5, to=1-5]
\end{tikzcd}\]
 where in $\mathbf{E^*}$ we identify $\wedge^iF_{2,3}^* \cong \wedge^{e-3-i}F_{2,3}$ by the canonical isomorphism with appropriate signs, and $$\delta_{e-4}^*(f_i\otimes g_1^jg_2^{e-6-j} )= X_ig_1^{j+1}g_2^{e-6-j}+X_{i+1}g_1^jg_2^{e-5-j}.$$ 
 Note that $A_{2,3}$ is a $2 \times e-3$ matrix and by \Cref{lem:height of m and m+1} has generic height $e-4.$ Thus, both $\mathbf E$ and $\mathbf E^\ast$ are exact. To show that $\psi$ extends to a map of complexes, it suffices to check that the image of $\psi \circ \delta_{e-4}^\ast$ is in ${\bf I}_2(A_{2,3})$. Indeed, 
 \begin{eqnarray*}
 \psi\circ \delta^\ast_{e-4} ( f_i\otimes g_1^jg_2^{e-6-j})&=& \psi (X_ig_1^{j+1}g_2^{e-6-j}+X_{i+1}g_1^jg_2^{e-5-j})\\
 &=&(-1)^j[X_1(X_{i+1}X_{j+7}-X_iX_{j+8})+X_4(X_{i}X_{j+5}-X_{i+1}X_{j+4})] \in {\bf I}_2(A_{2,3}). 
\end{eqnarray*}

Thus, $\psi$ extends to a map of complexes ${\bf E^*} \to {\bf E}$ and one can verify that
\begin{equation}\label{eq:check psi1}
    \psi_1(f_i\otimes g_1^jg_2^{e-6-j})= X_1 C + X_4 D, \q \text{ for some } C, D \in \wedge^2 F_{2,3}.
\end{equation}
Finally, the mapping cone $\psi$ is exact and is the resolution of  $\IM{2,3}$. 
\end{proof}

\begin{corollary}\label{Betti3}
    The resolution in \Cref{thm:res of I23} is minimal.  Hence, for $1\le t\le e-4,$ the Betti numbers of $k[\SM{2,3}]$ are 
    $$\beta_t= \frac{t(e-1)}{e-t-2}{e-3\choose t+1},$$
    and $\beta _0 =\beta_{e-3}= 1.$
\end{corollary}
    
\begin{proof}
    Let $\psi: \mathbf E^\ast \to \mathbf E$ be the map of complexes as in the Theorem \ref{thm:res of I23}.  If we choose the same bases for $F_{2,3}$ and $G$, the non-zero entries in the matrix of $\psi_1$ are of degree  $1$ in the standard grading as can be seen from \Cref{eq:check psi1}. Indeed, each entry is $X_1$ or $X_4$.  
    Now, the nonzero entries in the matrices of the resolutions $\mathbf E$ and $\mathbf E^*$ are again of degree $1$ (excluding $\delta_1$ and $\delta^*_{e-4}$).  It follows that the non-zero entries in the matrices of $\psi_i, i\ge 1$ must all be of degree at least $1$, and hence, the resolution is minimal. 

    As $E_0 = \RM{2,3}$ and $E_0^* = \RM{2,3}^*,$ we get $\beta_0= \beta_{e-3} = 1$.  For $1\le t\le e-4$, the $t$-th free module in the resolution is $E_t\oplus E^*_{e-3-t}$. Hence the Betti numbers are 
    \[\beta_t = t{e-3\choose t+1}+(e-3-t){e-3\choose e-2-t} = \frac{t(e-1)}{e-t-2}{e-3\choose t+1}. \qedhere\] 
\end{proof}

\section{Case $\SM{3,4}$}  \label{sec:3,4}

As noted in \cite[Corollary 4.6]{DGS3}, $(m,n) = (3,4)$ is also a special case. In \cite{DGS3}, the authors proved that the minimal number of generators of the defining ideal, $\mu(\IM{3,4}) = {e-2 \choose 2}$ and also gave a minimal generating set. In this section, we identify a minimal generating set of $\IM{3,4}$ with the maximal minors of a $2 \times (e-2)$ matrix, compute a minimal free resolution and the Betti sequence of $k[\SM{3,4}].$

Set $X_e=X_0^2, X_{e+1}= X_0X_1, X_{e+2} = X_0X_2, X_{e+3} = X_1X_2,$ and $X_{e+4}= X_2^2.$
Let $A_{3,4}$ and $B_{3,4}$ be the matrices

$$A_{3,4} =  \begin{bmatrix}
    X_0   &X_1 & X_5 & \cdots &  X_{e-2} & X_{e-1} \\
    X_1   & X_2&  X_6  & \cdots    & X_{e-1} & X_e\\
\end{bmatrix}, \q 
B_{3,4} =  \begin{bmatrix}
    X_2   & X_{5} & \cdots & X_e& X_{e+1}\\
    X_5 & X_8  & \cdots & X_{e+3} & X_{e+4} 
\end{bmatrix}. $$

Let $A' =  \begin{bmatrix}
    X_e  & X_{e+1} & X_5 & \cdots &  X_{e-2} & X_{e-1} \\
    X_{e+1}  & X_{e+2} &  X_6  & \cdots & X_{e-1} & X_e\\
\end{bmatrix}.$ Then,
\[ I_2(A') =I_2 \left(  \begin{bmatrix}
    X_5 & \cdots &  X_{e-2} & X_{e-1} & X_0^2  & X_0X_1 \\
     X_6  & \cdots    & X_{e-1} & X_0^2 & X_0X_1  & X_0X_2&\\\end{bmatrix} \right) \subset I_2(A_{3,4}) \]
 and has height $e-4$. 
\begin{theorem}
The defining ideal $\IM{3,4}$ of $\SM{3,4}$ is  $$\IM{3,4}= {\bf I}_2(A_{3,4})+{\bf I}_2(B_{3,4}).$$
In particular, the minimal generators of $\IM{3,4}$ are the ${e-3\choose 2}$ size $2$ minors of $A_{3,4}$ and the $e-3$ size $2$ minors of $B_{3,4}$ that involve the first column. 
\end{theorem}  
\begin{proof}
   This follows from \cite[Subsection 5.6]{DGS3} in which a minimal generating set for $\IM{3,4}$ is given.
\end{proof}
   
\begin{theorem} \label{thm:res of I34}
Let $\displaystyle{F_{3,4} = \sum_{i=0,i \neq 2,3,4}^{e-1} \RM{3,4}f_i}$, $G = \RM{3,4}g_1\oplus \RM{3,4}g_2$ be free $\RM{3,4}$-modules, and $\phi: F_{3,4} \to G$ be given by the matrix $A_{3,4}$. Let $\mathbf{E}$ denote the Eagon-Northcott complex of $\phi$ and $\mathbf{E'}$ denote the Buchsbaum-Rim complex associated to the matrix $A'$  whose $i^{th}$ module is $E_i'=\wedge^{i+1}F_{3,4} \otimes D_iG^\ast$, for $1 \leq i \leq e-4,$ so that $E_0'=F_{3,4}$.
Define $\psi : F_{3,4} \to \RM{3,4}$ as follows
$$\psi( f_j) = \begin{cases}
    X_2X_{j+3}-X_5X_{j}, & 5\le j\le e-1\\
    X_2X_{j+e+3}-X_5X_{j+e}, & j=0,1.
\end{cases}$$
This map $\psi$ extends to a map of complexes  $\psi: \mathbf E'  \to \mathbf E$ whose mapping cone is the resolution of $\IM{3,4}$. 
\end{theorem}
\begin{proof} In this case, we have the following diagram:
\[\begin{tikzcd}
	{\mathbf E: 0} & {\overset{e-3}{\wedge} F_{3,4}\otimes D_{e-5}G^\ast} & \cdots & {\overset{2}{\wedge}F_{3,4}} & {\RM{3,4}} \\
	\E': 0 & {\overset{e-3}{\wedge} F_{3,4}\otimes D_{e-4}G^\ast} & \cdots & {\overset{2}{\wedge}F_{3,4}} \otimes DG^\ast & F_{3,4} 
	\arrow[from=1-1, to=1-2]
	\arrow["{\delta_{e-4}}", from=1-2, to=1-3]
	\arrow["{\delta_2}", from=1-3, to=1-4]
	\arrow["{\delta_1}", from=1-4, to=1-5]
	\arrow[from=2-1, to=2-2]
	\arrow["{\delta_{e-4}'}", from=2-2, to=2-3]
	\arrow["{\delta_2'}", from=2-3, to=2-4]
	\arrow["{\psi_1}"', dashed, from=2-4, to=1-4]
    \arrow["{\psi_{e-4}}", dashed, from=2-2, to=1-2]
	\arrow["{\delta_1'}", from=2-4, to=2-5]
	\arrow["\psi"', from=2-5, to=1-5]
\end{tikzcd}\]
By \Cref{lem:height of m and m+1} again, $A_{3,4}$ has generic height $e-4$ so both $\mathbf E$ and $\mathbf E'$ are exact. To show $\psi$ extends, it suffices to check that the image of $\psi \circ \delta_{1}'$ is in ${\bf I}_2(A_{3,4})$. Indeed, for $5\le i< j\le e-1$, 
\begin{eqnarray*}
\psi\circ \delta'_1 ( f_i\wedge f_j \otimes g_1)&=& \psi (X_if_j-X_jf_i)\\
 &=& X_i(X_2X_{j+3}-X_5X_{j})-X_j (X_2X_{i+3}-X_5X_{i})\\
 &=& X_2(X_i X_{j+3}-X_jX_{i+3})\in {\bf I}_2(A_{3,4}),
\end{eqnarray*} 
and  
\begin{eqnarray*}
  \psi\circ \delta'_1 ( f_i\wedge f_j \otimes g_2)&=& \psi (X_{i+1}f_j-X_{j+1}f_i)\\
 &=& X_{i+1}(X_2X_{j+3}-X_5X_{j})-X_{j+1} (X_2X_{i+3}-X_5X_{i})\\
 &=& X_2(X_{i+1} X_{j+3}-X_{j+1}X_{i+3})+X_5(X_iX_{j+1}-X_{i+1}X_j)\in {\bf I}_2(A_{3,4}),   
\end{eqnarray*}
with 
\begin{equation}\label{eq:psi1 34 ij}
 \psi_1 ( f_i\wedge f_j \otimes g_1) = X_2 \left( \sum_{u=0}^{2} f_{i+u} \wedge f_{j+2-u} \right),  \,    
 \psi_1 ( f_i\wedge f_j \otimes g_2) = X_2 \left( \sum_{u=1}^{2} f_{i+u} \wedge f_{j+3-u} \right)+X_5(f_i\wedge f_j).
 \end{equation} 
When $i =0 \in \{0,1\},$ and $j\ge 5$, 
\begin{eqnarray*}
\psi\circ\delta'_1 ( f_i\wedge f_j \otimes g_1)&=& \psi (X_iX_0f_j-X_jf_i)\\
 &=& X_0X_i(X_2X_{j+3}-X_5X_{j})-X_j (X_2X_{e+i+3}-X_5X_{e+i}) \\
 &=& X_2(X_0X_i X_{j+3}-X_jX_{i+e+3}) \in {\bf I}_2(A_{3,4}),
 \end{eqnarray*}
 and 
 \begin{eqnarray*}
 \psi\circ\delta'_1 ( f_i\wedge f_j \otimes g_2)&=& \psi(X_{i+1}X_0f_j-X_{j+1}f_i)\\
 &=& X_{i+1}X_0(X_2X_{j+3}-X_5X_{j})-X_{j+1} (X_2X_{i+e+3}-X_5X_{e+i})\\
 &=& X_2(X_0X_{i+1} X_{j+3}-X_{j+1}X_{i+e+3})+X_5(X_{e+i}X_{j+1}-X_0X_{i+1}X_j)\in {\bf I}_2(A_{3,4}),
 \end{eqnarray*}
 with
\begin{align} 
\psi_1(f_i \wedge f_j \otimes g_1) &= X_0X_2(f_i \wedge f_{j+2}) + X_{i+1}X_2 \left( \sum_{u=0}^{1} f_{u} \wedge f_{j+1-u} \right), 
\label{eq:psi1 g1 34 0j} \\ 
\psi_1(f_i \wedge f_j \otimes g_2) &= X_{i+1}X_2 \left( \sum_{u=0}^{1} f_{u} \wedge f_{j+2-u} \right) + X_0X_5(f_i \wedge f_j). \label{eq:psi1 g2 34 0j}
\end{align} 

Finally, 
  \begin{eqnarray*}
 \psi\circ \delta' ( f_0\wedge f_1 \otimes g_1)&=& \psi (X_0^2f_1-X_0X_1f_0)\\
 &=& X_0^2(X_2X_{e+4}-X_5X_{e+1})-X_0X_1 (X_2X_{e+3}-X_5X_{e})\\
 &=& X_0X_2^2(X_0 X_2-X_1^2)\in {\bf I}_2(A_{3,4}),
 \end{eqnarray*} 
 and
 \begin{eqnarray*}
 \psi\circ \delta'( f_0\wedge f_1 \otimes g_2)&=& \psi X_0(X_{1}f_1-X_{2}f_0)\\
 &=& X_0X_1(X_2X_{e+4}-X_5X_{e+1})-X_0X_2(X_2X_{e+3}-X_5X_e)\\
 &=& X_0^2X_5(X_0X_2 - X_1^2)\in {\bf I}_2(A_{3,4}),
 \end{eqnarray*}
 with 
\begin{equation} \label{eq:psi1 34 01}
\psi_1(f_0 \wedge f_1 \otimes g_1) = X_0X_2^2(f_0 \wedge f_1), \q \psi_1(f_0 \wedge f_1 \otimes g_2) = X_0^2X_5(f_0 \wedge f_1).
\end{equation}

Hence, the mapping cone of $\psi$ is exact and is the resolution of $\IM{3,4}$.
\end{proof}

\begin{corollary}\label{cor:I34}
    The resolution in \Cref{thm:res of I34} is minimal.  Hence, the Betti numbers of $k[\SM{3,4}]$ are 
    $$\beta_t= t{e-2 \choose t+1}.$$
    \end{corollary}

    \begin{proof}
    Let $\psi: \mathbf E'\to \mathbf E$ be the map of complexes as in the theorem \Cref{thm:res of I34}.  If we choose bases for $F_{3,4}$ and $G$ as above, the non-zero  entries in the matrix of $\psi_1$ are of degree at least $1$ (see  \eqref{eq:psi1 34 ij}, \eqref{eq:psi1 g1 34 0j}, \eqref{eq:psi1 g2 34 0j} and \eqref{eq:psi1 34 01}).  Now, the nonzero entries in the matrices of the resolutions $\mathbf E$ and $\mathbf E'$ from $E_2$ onwards are again of degree $1$, except when they involve $X_0$ when they can be of degree $2$.  Since the entries that involve $X_0$ in $\mathbf E'$ are also of degree $2$, there is no danger of getting $X_0^2$ as the coefficient of any basis element in the composition $\psi_{i-1} \circ \delta'$.   It follows that the nonzero entries in the matrices of $\psi_i, i\ge 1$ must all be of degree at least $1$ and hence, the resolution is minimal. 
    
    For $1 \leq t \leq e-3,$ the $t$-th module in the resolution is $E_t \oplus E'_{t-1}.$ Thus, the Betti numbers are 
    \[\beta_t = t{e-3 \choose t+1}+ t {e-3 \choose t}= t{e-2 \choose t+1}. \qedhere \]
\end{proof}

\begin{remark}  In this case all the Betti Numbers are indeed the same as that of an ideal of $2\times 2$ minors of a $2\times e-2$ matrix. Indeed, ${\bf I}_2( B_{3,4} )\subset \IM{3,4} $ is the ideal of $2\times 2$ minors of a $2\times e-2$  matrix.   However, ${\bf I}_2(B_{3,4})$ is not prime and hence, is strictly contained in $\IM{3,4}$. 
\end{remark}

\section{Trailing Betti numbers}  \label{sec:conj}

Recall that in our study of Sally type semigroups in this paper, we encounter two Gorenstein semigroups, namely $\SM{1}$ and $\SM{2,3}.$ Since Betti numbers form an important tool to study the homological properties of a ring, it is worthwhile to explore the Betti sequence of other Sally type semigroups in relation to the Betti sequence of the corresponding Gorenstein ones. In this section, we propose various conjectures on the Betti numbers of $\SM{m,n}$ as a function of $m$ and $n$. 

For the rest of this section, we fix the multiplicity $e$, and set $\beta_j(m,n):=\beta_j(\SM{m,n})$.

\begin{conjecture}
    For any $m \geq 2,$ the Betti numbers of $\SM{m}$ trail those of $\SM{1}.$ That is, 

$$\beta_j(m) = \begin{cases}
    \beta_j(1), &1 \leq j \leq m-2\\
    \beta_j(1) + {e-m \choose j+1-m}, & m-1 \leq j \leq e-3, \\
{e-m \choose j+1-m}=e-m, & j=e-2.
\end{cases}$$
\end{conjecture}

 \begin{conjecture}
Let $n\geq 6$. The Betti numbers of $\SM{2,n}$ trail the Betti numbers of $\SM{2,3}$. To be more precise,
\begin{align*}
    \beta_j(2,n) & =\beta_j(2,3), \q \forall \; j\leq n-5, \\
    \text{ and } \q \beta_{n-4}(2,n) & =\beta_{n-4}(2,3)+1.
\end{align*}
Moreover, one can say that $\SM{2,e-1}$ is the closest to being Gorenstein in the given family of rings. We have the following observation,
     \begin{equation*}
         \beta_j(2,e-1) = \begin{cases}
             \beta_j(2,3), & j \leq e-6 \\
             \beta_j(2,3)+1, & j=e-5 \\
             \beta_j(2,3)+3, & j=e-4 \\
             \beta_j(2,3)+2=3, & j=e-3.
         \end{cases}
     \end{equation*}
\end{conjecture}

In Section \ref{sec:3,4}, we studied the Betti sequence of $\SM{3,4}.$ From \cite[Corollary 4.6]{DGS3} recall that $(2,4)$ and $(2,5)$ are also special cases.
\begin{conjecture}
The Betti numbers of $\SM{2,4}$ and $\SM{2,5}$ satisfy the same formula as those of $\SM{3,4}$. 
\[ \beta_j(2,4) = \beta_j(2,5) = j {e-2 \choose j+1}, \ \forall j. \]
\end{conjecture}

\begin{conjecture}
    For any $m \geq 4,$ the Betti numbers of $\SM{m,m+1}$ and $\SM{m+1,m+2}$ are the same for $\beta_1,\beta_2,\ldots,\beta_{m-3}.$ In other words, $$\beta_j(n,n+1)=\beta_j(m,m+1), \ \forall \ n\geq m \text{ and } \ \forall j \leq m-3.$$
\end{conjecture} 

Another peculiar case is the one with $m=1.$ Note that when $n=e-1,$ the width of the semigroup reduces to $e-2.$ 
\begin{conjecture} 
We first conjecture a formula for the Betti numbers for $\SM{1,e-1}.$ For $0 \leq j \leq e-5,$
\begin{align*}
\beta_j(1,e-1) &=j {e-3 \choose j+1} + (e-4-j){e-3 \choose e-2-j}, \\
\beta_{e-4}(1,e-1) &=j {e-3 \choose j+1} + (e-3-j){e-3 \choose j} = (e-4) + (e-3) = 2e-7, \\
\beta_{e-3}(1,e-1) &=j {e-3 \choose j+1} + (e-1-j){e-3 \choose j} = 2.
\end{align*}
For $m\geq 3,$ the first $m-2$ Betti numbers of $\SM{m,e-1}$ are exactly the same as the Betti numbers of $\SM{1,e-1}.$ That is, 
$$\beta_j(1,e-1)=\beta_j(m,e-1), \q \forall \ j \leq m-2.$$
\end{conjecture}

\begin{acknowledgment*}{\rm
Hema Srinivasan is supported by grants from Simons Foundation. We are grateful to the software systems Singular \cite{Singular}, Macaulay2 \cite{M2}, and GAP \cite{NumericalSgps} for serving as an excellent source of inspiration.
}\end{acknowledgment*}

\printbibliography
\end{document}